\documentclass[a4paper]{amsart}

\usepackage{dsfont}
\usepackage{xspace}
\usepackage{algorithmic}
\usepackage{graphicx}
\usepackage{url}

\theoremstyle{plain}
\newtheorem{theorem}{Theorem}[section]

\newtheorem{lemma}[theorem]{Lemma}
\newtheorem{corollary}[theorem]{Corollary}
\newtheorem{proposition}[theorem]{Proposition}

\theoremstyle{definition}

\newtheorem{definition}[theorem]{Definition}

\newenvironment{myalg}[1]{%
  \begingroup \rm
  \vspace{2ex}\hrule\vspace{1ex}
  \refstepcounter{theorem}
  \noindent\textbf{Algorithm \thetheorem.\ 
    #1}\vspace{1ex}\hrule\vspace{1ex}}{\vspace{1ex}\hrule\vspace{2ex}\endgroup}

\newcommand{\CC}{\ensuremath{\mathds{C}}}

\newcommand{\KK}{\CC}
\newcommand{\cocoa}{{\hbox{\rm C\kern-.13em o\kern-.07em C\kern-.13em
      o\kern-.15em A}}\xspace}
\newcommand{\singular}{\textsc{Singular}\xspace}
\newcommand{\divides}{\mathop{|}}

\begin{document}
\author[A. Dickenstein]{Alicia Dickenstein}
\address{Depto. de Matem\'atica, FCEN, Universidad de Buenos Aires, and IMAS-CONICET
  Argentina}
\email{alidick@dm.uba.ar}

\author[E. A. Tobis]{Enrique A. Tobis$^\dag$}
\address{Depto. de Matem\'atica, FCEN, Universidad de Buenos Aires,
  Argentina}
\email{etobis@dc.uba.ar}
\thanks{$^\dag$ Corresponding author. Present address: Children's Hospital Boston, 300 Longwood Ave., Boston, MA, USA}

\title{Independent Sets from an Algebraic Perspective} \date{}
\thanks{Both authors were partially supported by UBACYT X064, CONICET
  PIP 112-200801-00483 and ANPCyT, Argentina}

\begin{abstract}
  In this paper, we study the basic problem of counting independent
  sets in a graph and, in particular, the problem of counting
  antichains in a finite poset, from an algebraic
  perspective. We show that neither independence polynomials of bipartite
  Cohen-Macaulay graphs nor  Hilbert series of
  initial ideals of radical zero-dimensional complete intersections
  ideals, can be evaluated in polynomial time, unless $\#P=P$. 
Moreover, we present a family of radical
  zero-dimensional complete intersection ideals $J_P$ associated to a
  finite poset $P$, for which we describe a universal Gr\"obner basis.
  This implies that the bottleneck in computing the dimension of the
  quotient by $J_P$ (that is, the number of zeros of $J_P$) using
  Gr\"obner methods lies in the description of the standard
  monomials.
\end{abstract}

\maketitle

\section{Introduction}
\label{sec:intro}

We approach the basic problem of counting independent
sets in a graph and, in particular, the problem of counting antichains
in a finite partially ordered set, from an algebraic perspective. We
derive structural considerations and complexity results.

The use of algebraic methods in the study of discrete problems, in
particular problems in graph theory, was pioneered by Richard Stanley
\cite{stanley} and L\'aszl\'o Lov\'asz \cite{lovasz}, from the
combinatorics side, and J\"urgen Herzog, Takayuki Hibi, Aron Simis,
Wolmer Vasconcelos and Rafael Villarreal \cite{herhibi,
  simisvasvilla,monal} from the commutative algebra side. The
enumeration of independent sets has been approached using Reverse
Search (\cite{eppstein}), the Belief Propagation heuristic
(\cite{count2}) and Binary Decision Diagrams (\cite{count1}), to name
a few techniques.

The main algebraic object we will use is the Hilbert
Series of the initial monomial ideal associated with a graph. The
problem of computing a Hilbert Series is NP-Complete
(\cite{davemike}). There is a standard algorithm (first proposed in
\cite{moramoller}) for computing the Hilbert Series of a quotient
$\KK[\mathbf{x}]/I$, where $I$ is a homogeneous ideal in
$\KK[\mathbf{x}]$. There are some classes of
ideals for which this algorithm finishes in time polynomial in the
input, e.g. Borel (\cite{davemike}) and Borel-type ideals
(\cite{hashemi}).  Open computer algebra systems (\cocoa
\cite{CocoaSystem}, \singular \cite{GPS}, Macaulay2 \cite{M2})
implement the standard algorithm in subtly different ways. 
We suggest \cite[Ch. 5]{cca2} as a general
reference on Hilbert Series.

The connection between independent sets and commutative
algebra is spearheaded by the following construction.
\begin{definition}
  Let $G = (V,E)$ be a graph, with $V = \{v_1,\allowbreak
  \ldots,\allowbreak v_n\}$. The \emph{edge ideal}
  (\cite{simisvasvilla,monal}) $I'_G \subseteq \KK[x_1,\ldots,x_n]$ of
  $G$ is defined as
\begin{equation}
  \label{eq:1}
  I'_G = \langle x_ix_j, \text{ for all } (v_i,v_j) \in E \rangle.
\end{equation}
\end{definition}
This ideal links independent sets in $G$ and certain monomials. If
$x^\alpha$ is a monomial \emph{not} in $I'_G$ (termed a \emph{standard
monomial}), then it encodes an
independent set $S$ of $G$ in this way:
\begin{equation}
  \label{eq:2}
  v_i \in S \Leftrightarrow x_i \mathop{|} \mathbf{x}^\alpha.
\end{equation}
This encoding is not one-to-one. For example, the monomials $x_1$ and
$x_1^2$ represent the same independent set: $\{v_1\}$. We introduce a
slightly modified version of $I'_G$, with which we obtain a bijective
encoding.
\begin{definition}
  Let $G = (V,E)$ be a graph. We define the \emph{modified edge ideal}
  $I_G$ of $G$ as
  \begin{equation}
    \label{eq:3}
    I_G = I'_G + \langle x_i^2, \text{ for all } v_i\rangle.
  \end{equation}
\end{definition}
Notice that $I_G$ is zero-dimensional (the origin is the only
root), and that its standard monomials are square-free,
with the degree of a monomial equal to the size of the
corresponding independent set. The number of independent (or stable)
sets in $G$ thus coincides with the $k$-vector space dimension of the
quotient of the polynomial ring in $n$ variables over any field $k$ by
the ideal $I_G$. This dimension is computed in
\cite{CocoaSystem,M2,GPS} using the additivity of the Hilbert function
in short exact sequences.

In Section~\ref{sec:hilbert} we recall the definition of the Hilbert
function (see~\eqref{eq:4}) and we analyze the instantiation of the
standard algorithm for computing the Hilbert Series for the ideals
$I_G$. Our main result in this section shows that the recursive
calls simply correspond to counting independent sets of $G$ that
contain a pivot vertex, and those that do not contain it. In
Section~\ref{sec:posets-gb}, we turn our attention to the subproblem
of counting the antichains of a finite poset. We
present 
the universal reduced Gr\"obner Basis for a family of zero-dimensional
radical ideals derived from posets. 
In Section~\ref{sec:cohen}, we
specialize our study in the case of Cohen-Macaulay bipartite graphs, corresponding to
Cohen-Macaulay ideals $I'_G$. Using the characterization in~\cite{herhibi}, we show
that counting independent sets in such graphs is equivalent to evaluating at 
$2$ the independence polynomial of the comparability graph of a general finite
poset. Section~\ref{sec:complex} contains our complexity study. We prove that
antichain polynomials cannot be evaluated in polynomial time at any non-zero
rational number $t$ unless $P= \#P$. When combined with the algebraic results
from the previous sections we deduce Corollaries~\ref{cor:Hilberthard} 
and~\ref{cor:CMhard} on the intractability of the evaluation of Hilbert
Series of initial ideals of zero-dimensional complete intersections and
independence polynomials of Cohen-Macaulay bipartite graphs.
We close with a few experimental observations in Section~\ref{sec:exper}. 

\section{Counting independent sets via the computation of Hilbert
  Series}
\label{sec:hilbert}

We start by recalling a few definitions concerning Hilbert Series. Let
$M$ be a positively graded finitely generated $\CC[\mathbf{x}]$-module
(e.g. the quotient \(\KK[\mathbf{x}]/I_G \) for some graph \(G\)). We
can write
\[
M = \bigoplus_{0 \leq i} M_i,
\]
where $M_i$ is the subspace of $M$ of degree $i$. The Hilbert Function
(\(\mathit{HF}_M\)) of $M$ maps $i$ onto $\dim_\KK(M_i)$. The Hilbert
Series ($\mathit{HS}_M$) of $M$ is the generating function
\begin{equation}
  \label{eq:4}
  \mathit{HS}_M(z) = \sum_{0 \leq i}\mathit{HF}_M(i)\,z^i.
\end{equation}
If $M = \KK[\mathbf{x}]/I$ for a monomial ideal $I$, then
$\mathit{HF}_M(i)$ is the number of standard monomials of degree $i$ (that
is, monomials which \emph{are not} in $I$). 

If we take $I = I_G$ for some graph $G$, as we
mentioned in the introduction, $\mathit{HF}_M(i)$ is then the
number of independent sets of size $i$ in $G$. In this case, the
Hilbert Series of $\KK[\mathbf{x}]/I_G$ is a polynomial, called the
\emph{independence polynomial} of $G$. As usual, we denote this polynomial by
 $I(G,x)$. We refer the reader to \cite{indeppoly} for a
comprehensive survey of independence polynomias.

The standard algorithm for computing $\mathit{HS}_M$ hinges on the
following property. If we have a homogeneous exact sequence of
finitely generated graded $\CC[\mathbf{x}]$-modules
\begin{equation}
\label{eq:5}
0 \longrightarrow M' \longrightarrow M \longrightarrow M''
\longrightarrow 0,
\end{equation}
then
\begin{equation}
  \label{eq:6}
  \mathit{HS}_M(z) = \mathit{HS}_{M'}(z) + \mathit{HS}_{M''}(z).
\end{equation}
Given a finitely generated graded $\CC[\mathbf{x}]$-module $M$ and
$f \neq 0$ a homogeneous polynomial of degree $d$, we have the
following \emph{multiplication sequence}
\begin{equation}
  \label{eq:7}
  0 \longrightarrow [M/(0 :_M (f))](-d)
  \stackrel{\varphi}{\longrightarrow} M \longrightarrow M/fM
  \longrightarrow 0,
\end{equation}
where $\varphi$ is induced by multiplication by $f$. Here, $(0 :_M
(f)) = \{g \in M, \text{ such that } \allowbreak gf = \allowbreak
0\}$, and $(-d)$ induces a degree shift, so that $\varphi$ is a
homogeneous map of degree $0$. Rewriting equation (\ref{eq:6}) we
obtain
\begin{equation}
  \label{eq:8}
  \mathit{HS}_M(z) = \mathit{HS}_{M/fM}(z) + z^d\,\mathit{HS}_{(0:_M (f))}.
\end{equation}
The polynomial $f$ above is called a \emph{pivot}. 

Actually, the standard algorithm does not directly compute the Hilbert
Series. We can see in \cite[Theorem~5.2.20]{cca2} that in the case of the
modified edge ideal $I_G$, the Hilbert Series of $M =
\KK[x_1,\ldots,x_n]/I_G$ has the form
\begin{equation}
  \label{eq:9}
  \mathit{HS}_M = \frac{\mathit{HN}_M(z)}{(1-z)^n},
\end{equation}
where $\mathit{HN}_M(z)$ is called the \emph{Hilbert Numerator}. The
algorithm computes $\mathit{HN}_M(z)$, and the series is then obtained
by dividing it by $(1-z)^n$.

We reproduce the algorithm for computing the Hilbert Numerator of a
monomial ideal (see \cite[Theorem~5.3.7]{cca2}).


\begin{myalg}{Algorithm to compute the Hilbert Numerator of a monomial
    ideal $I$ (called \texttt{HN}).}
  \begin{algorithmic}[1]
    \REQUIRE A set of minimal monomial generators for the ideal $I$.

    \ENSURE The Hilbert Numerator of $\KK[\mathbf{x}]/I$.

    \IF{the minimal generators of $I$ are pairwise coprime}

    \RETURN $\prod_{i=1}^s(1-z^{d_i})$, where $d_i$ is the degree of
    the $i$-th generator of $I$.

    \ELSE 

    \STATE Choose a monomial $p$ as pivot.\label{alg1:step:choice}

    \STATE $f_1 \leftarrow \mathtt{HN}(I:p)$.

    \STATE $f_2 \leftarrow \mathtt{HN}(I + p)$.

    \RETURN $z^{\deg(p)}f_1(z) + f_2(z)$.

    \ENDIF
  \end{algorithmic}
  \label{alg1}
\end{myalg}

Notice that the sets of generators of $I'_G$ and $I_G$ described
in~\eqref{eq:1} and~\eqref{eq:3} are minimal. The process of obtaining
minimal sets of generators for the ensuing recursive calls can be
optimized by performing careful interreductions.

The choice of pivot must satisfy one condition. Namely, 
\begin{equation}
  \label{eq:10}
  \sum \deg(I:p) < \sum\deg(I) \quad \text{and} \quad \sum \deg(I+p) <
  \sum\deg(I).
\end{equation}
Here, $\sum\deg(I)$ denotes the sum of the degrees of all the minimal
monomial generators of $I$. Intuitively, this condition says that the
recursive calls are made on ``smaller'' ideals, and shows that the
algorithm terminates.

The program \cocoa implements this algorithm, and uses a certain
strategy for the choice of pivot in
step~\ref{alg1:step:choice}. First, it chooses any variable $x_i$
appearing in the most number of generators of $I_G$. Then it picks two
random generators containing that variable. The pivot is the highest
power of $x_i$ that divides \emph{both} random generators.

We present a specialized version of Algorithm~\ref{alg1}, suited for
the computation of the Hilbert Series of $\KK[\mathbf{x}]/I_G$ for
any graph $G$.

\begin{theorem}
  \label{thm:specialversion}
  Let $I_G$ be the modified edge ideal of a graph $G$. The general
  algorithm for computing the Hilbert Series of $\KK[\mathbf{x}]/I_G$
  has the specialized version presented in
  Algorithm~\ref{alg:special}.

  This algorithm has an obvious graphical interpretation. The choice
  of step~\ref{alg2:step:choice} corresponds to choosing a node $v$ of
  the graph. The recursive calls of step~\ref{step:combination}
  correspond to counting the independent sets of $G$ that contain $v$
  ($\text{HS}_{\text{Colon}}$) and those that do not contain $v$
  ($\text{HS}_{\text{Plus}}$).

  \begin{myalg}{Specialized algorithm to compute the $\mathit{HS}$ of
      $\KK[\mathbf{x}]/I_G$.}
    \begin{algorithmic}[1]
      \REQUIRE The list $L$ of minimal monomial generators of $I_G$
      described in~(\ref{eq:3}).
      
      \ENSURE The Hilbert Series of $\KK[\mathbf{x}]/I_G$.
      
      \IF{$L$ consists only of variables and squares of
        variables}\label{step:if}
      
      \RETURN $(1+z)^k$, where $k$ is the number of variables which
      appear squared in $L$.\label{step:compint}

      \ELSE

      \STATE Choose a variable $x_i$ that appears squared in
      $L$.\label{alg2:step:choice}
      
      \STATE Colon $\leftarrow$ a minimal set of monomial generators
      of $(\langle L \rangle: x_i)$.\label{alg:special:colon}
      
      \STATE Plus $\leftarrow$ a minimal set of monomial generators of
      $\langle L, x_i \rangle$.\label{alg:special:plus}
      
      \RETURN $z\text{HS}_{\text{Colon}}(z)$ +
      $\text{HS}_{\text{Plus}}(z)$\label{step:combination}

      \ENDIF
    \end{algorithmic}
    \label{alg:special}
  \end{myalg}
\end{theorem}
\begin{proof}
  Algorithm~\ref{alg:special} differs from Algorithm~\ref{alg1} in two
  key steps. In step~\ref{step:if}, the special version does not check
  coprimality, as is done in Algorithm~\ref{alg1}. The other
  difference is in step~\ref{alg2:step:choice}: The specialized
  version chooses a variable, instead of an arbitrary monomial.

  We make a claim that helps us understand why this specialized
  version is correct. In every call to the algorithm, each of the $n$
  variables appears in $L$ raised to the first or to the second
  power. Furthermore, in each call, $L$ contains only the powers just
  mentioned and the ``edge monomials'' $x_ix_j$ of $G$ such that both
  $x_i$ and $x_j$ appear squared in $L$. This leads to an obvious
  graphical interpretation: The list $L$ represents the subgraph of
  $G$ induced by those variables that appear squared in $L$.

  We prove the correctness of the algorithm by showing that the choice
  of a pivot in Algorithm~\ref{alg1} must always yield a variable when
  applied to a modified edge ideal, and that the claim of the previous
  paragraph is true.

  When the algorithm is originally invoked, every variable appears
  squared in $L$. Besides the squares of variables, $L$ contains the
  ``edge monomials'' $x_ix_j$ for every edge $(i,j)$ of $G$. This
  proves that the claim above holds in the first call.

  Assuming that the elements of $L$ have the structure we claim, let
  us show that any choice of pivot yields a variable. 
  Suppose that we employ any conceivable strategy for the choice of
  pivot, always subject to condition~\eqref{eq:10}. The pivot
  $p$ cannot be a multiple of any monomial in $L$. If it is, then
  $\text{Plus} = L$, and the decreasing total degree
  condition~\eqref{eq:10} is not satisfied. The pivot $p$
  must then be a product of variables that appear squared in $L$, but
  it must not be divisible by any ``edge monomial.'' Suppose that the
  pivot is the product of at least two variables. That is, $x_ix_j
  \divides p$, where $x_i^2$ and $x_j^2$ are in $L$, and $x_ix_j$ is
  not in $L$. Then Plus violates the decreasing total degree
  condition~\eqref{eq:10}, because it has the same generators
  as $L$, plus $p$. If $p = 1$, then $\text{Colon} = L$, and
  this violates the decreasing total degree condition. The only valid
  choice is then $p = x_i$, for some $x_i$ that appears squared in
  $L$.

  Once we know that the pivot is always a variable, we can show that
  the claim above holds for Plus and for Colon. In doing so, we also
  explain the second part of the theorem.

  The list of minimal monomial generators for Plus contains all the
  variables that were raised to the first power in $L$. Furthermore,
  it must also contain the pivot $x_i$. The square of $x_i$ is not in
  Plus, because Plus is minimal, and the ``edge monomials'' that
  contained $x_i$, are not present in Plus. The rest of the generators
  in $L$ are unaffected. Therefore, we have that every variable
  appears in Plus either squared or raised to the first power, as we
  wanted to show. Plus corresponds to the graph obtained by removing
  the node that corresponds to $x_i$ and all the edges incident with
  it.
  
  The analysis of Colon is somewhat similar. To obtain a minimal set
  of monomial generators, we just cross out the pivot $x_i$ from every
  generator in $L$ that contains it, and then eliminate multiples. If
  we had an ``edge monomial'' $x_ix_j$, then $x_j$ is in
  Colon. Therefore, the square of $x_j$ is no longer a generator, and
  all the ``edge monomials'' containing $x_j$ are also missing from
  Colon. Again, every variable appears either squared or raised to the
  first power. In this case, we remove the node corresponding to
  $x_i$, all its adjacent nodes and all the edges incident with $x_i$
  or with any node adjacent to $x_i$.

  Let $v$ be the node of $G$ associated with the pivot $x_i$. The
  combination step of the algorithm reflects the meaning of Colon and
  Plus: The independent sets of $G$ are those of Plus (i.e. those
  \emph{do not} that contain $v$) and those of Colon (i.e. those that
  contain $v$).

  The algorithm terminates when there are no more ``edge
  monomials''. Since all the generators are variables, or squares of
  variables, then they are pairwise coprime and satisfy the stopping
  criterion of Algorithm~\ref{alg1}.

  A note is in order about the value returned in the base
  case. Algorithm~\ref{alg1} returns
  \begin{equation}
    \label{eq:13}
    \prod_{i=1}^n(1-z^{d_i}),    
  \end{equation}
  where $d_i$ is the degree of the $i$-th generator. Since in the
  specialized case the generators are of the form $x_i$ or $x_i^2$,
  expression~\eqref{eq:13} has the form
  \begin{equation}
    \label{eq:14}
    (1-z)^n (1+z)^k,
  \end{equation}
  where $k$ is the number of variables that appear squared in
  $L$. According to formula~\eqref{eq:9}, the value returned by
  Algorithm~\ref{alg:special} is the Hilbert Series of
  $\KK[x_1,\ldots,x_n]/I_G$.

  All these observations show that the graphical interpretation is
  accurate and that the specialized version is indeed correct.


\end{proof}


\section{Partially ordered sets and Gr\"obner Bases}
\label{sec:posets-gb}

In this section, we study a family of zero-dimensional radical
complete intersection polynomial ideals associated with posets, first proposed in
\cite{cd}. 

Recall that a \emph{poset} (or \emph{partially ordered set}) is a set $P$,
together with a  (partial order) relation $\leq$ satisfying
\begin{itemize}
\item $a \leq a$, for all $a \in P$.
\item $a \leq b$ and $b \leq a$ implies $a = b$, for all $a$ and $b$
  in $P$.
\item $a \leq b$ and $b \leq c$ implies $a \leq c$ for all $a$, $b$
  and $c$ in $P$.
\end{itemize}
Two elements $a$ and $b$ of $P$ are
\emph{comparable} if $a \leq b$ or if $b \leq
a$. Otherwise, they are \emph{incomparable}. We will usually 
just write $P$ and drop the partial order relation from the notation.

We can associate to a poset $P$ its comparability graph.

\begin{definition}
Let $P$ be a poset. The comparability graph $G(P)$ has
the set $P$ as nodes and there is an edge between two different nodes $a$ and $b$ if
and only if $a,b$ are comparable in $P$. 
\end{definition}
A subset $S$ of a poset 
$P$ is an \emph{antichain} if all the elements of $S$ are pairwise
incomparable in $P$. We write ${\mathcal A}(P)$ for the set of antichains of
$P$. Note that $S \in {\mathcal A}(P)$ if and only if $S$ is an independent 
set of $G(P)$.

\begin{definition}
For any poset $P$ we define the \emph{antichain polynomial} $A(P,x)$
by
\[A(P,x) \, = \, I(G(P),x).\]
\end{definition}
Thus, the $k$-th coefficient of $A(P,x)$ equals the number of antichains of $P$ with $k$ elements and 
the cardinal $|\mathcal A (P)|$ is given by the evaluation $A(P,1)$.

\medskip

Given a finite poset $(P,\leq)$, we define a polynomial ideal
$J_P \subset \CC[x_1,\ldots,x_n]$ by:
\begin{equation}
  \label{eq:15}
  J_P = \langle  x_i - x_i \prod_{v_j \leq v_i} x_j, \text{ for all }
  v_i \in P \rangle.
\end{equation}

\begin{lemma}
  \label{lem:zeroorone}
  Let $P$ be a finite poset. Then the elements of
  $V(J_P)$ are strings of $0$'s and $1$'s.
\end{lemma}
\begin{proof} Let $a \in V(J_P)$.
  Suppose that an element $v_i \in P$ is minimal. Then $x_i -
  x_i^2 \in J_P$, hence $a_i$ is $0$ or $1$. Now, take any $v_i$, and
  assume that for every $v_j < v_i$ we know that $a_j$ is $0$
  or $1$. Note that $x_i - x_i \prod_{v_j \leq v_i} x_j = x_i (1 - x_i \prod_{v_j < v_i} x_j)$.
  If any  $a_j$ is $0$, then $a_i$ must be $0$ too. If all $a_j$ are $1$ then $a_i (1-a_i) =0$.
\end{proof}

Moreover, we have:

\begin{theorem}[\cite{cd}]
  \label{thm:varequalsanti}
  For any finite poset $P$, $J_P$ is a radical
  zero-dimensional ideal. Then, it has a finite number of simple
  zeros. Furthermore,
  \begin{equation}
    \label{eq:16}
    |V(J_P)| = |{\mathcal A}(P)|.
  \end{equation}
\end{theorem}

We now show that we can present $J_P$ as a zero-dimensional complete 
intersection by means of generators of lower degree.
A standard alternate way of dealing with a poset $P$ is
to look at the \emph{cover} relation. Given $a$ and $b$
in $P$, we say that $a \prec b$ (read ``$b$ covers $a$'') if and only
if $a < b$ and there is no $c \in P$ such that $a < c < b$. Using this
relation we define the ideal
\begin{equation}
  \label{eq:17}
  J'_P = \langle  x_i - x_i \prod_{v_j \preceq v_i} x_j, \text{ for all }
  v_i \in P \rangle.
\end{equation}

\begin{lemma}
  Let $P$ be a finite poset. Then
  \begin{equation}
    \label{eq:18}
    J_P = J'_P.
  \end{equation}
\end{lemma}
\begin{proof}
  It is straightforward to see that the varieties of $J_P$ and $J'_P$
  coincide. We show that $J'_P$ is radical. Since we already know that
  $J_P$ is radical, this proves the equality.

It is enough to prove that the square-free polynomial
  $x_i-x_i^2$ is in $J'_P$ for all
  $v_i \in P$. We know this to be true for the minimal elements of
  $P$, by the very definition of $J'_P$. Suppose we have a non-minimal
  element $v_i$ in $P$. Let $v_{j_1},\ldots,v_{j_r}$ be the
  elements such that $v_{j_k} \prec v_i$, and assume that $x_{j_l} -
  x_{j_l}^2$ is in $J'_P$ for all $l$. First, we observe that
  $x_ix_{j_l} - x_i$ is in $J'_P$ for all $l$. Indeed,
  \[
  (x_{j_l} - 1)(x_i-x_i^2 \prod_{k=1}^r x_{j_k}) - \left(x_i^2
  \prod_{\substack{k=1\\k\neq l}}^r x_{j_k}\right)(x_{j_l} - x_{j_l}^2) =
  x_ix_{j_l} - x_i.
  \]
  Now, consider the following step:
  \begin{equation*}
    (x_i - x_i^2 \prod_{k = 1}^r x_{j_k}) - (x_i - x_ix_{j_r}) x_i
    \prod_{k = 1}^{r-1}x_{j_k}
    = x_i - x_i^2\prod_{k = 1}^{r-1}x_{j_k}.
  \end{equation*}
  Since $(x_i - x_i^2 \prod_{k = 1}^r x_{j_k})$ and $(x_i -
  x_ix_{j_r})$ are in $J'_P$, we have that $x_i - x_i^2\prod_{k =
    1}^{r-1}x_{j_k}$ is also in $J'_P$. If we apply this procedure
  repeatedly, we eliminate variables from the product, and eventually
  find that $x_i - x_i^2$ is in $J'_P$.
\end{proof}

We now take Theorem~\ref{thm:varequalsanti} one step further, and give
an explicit bijection between ${\mathcal A}(P)$ and $V(J_P)$.

\begin{proposition}
  Let $P$ be a finite poset. Define the
  function $f : V(J_P) \to {\mathcal A}(P)$ by
  \begin{equation*}
    \label{eq:19}
    f(a) = \{v_i \in P, \text{ such that } a_i = 1
    \text{ and } a_j = 0 \text{ for all } v_j > v_i \}.
  \end{equation*}
  The map $f$ is bijective, and its inverse $g : {\mathcal A}(P) \to
  V(J_P)$ is defined by
  \begin{equation*}
    g(S)= a', \text{ where } a'_i = 1 \text{ if }
    \exists\, v_j \in S \text{ such that } v_i \leq v_j \text{ and }
    a'_i =0 \text{ otherwise}.
  \end{equation*}
\end{proposition}

\begin{proof}
It is clear from the definition of $f$ that no  pair of elements of the subset
$f(a)$ can be comparable for any $a \in V(J_P)$, that is, that $f(a)$ is
indeed an antichain. Reciprocally, let $S$ be an antichain of $P$ and let
$a' = g(S)$. We need to see that $ a'_i (1 - a'_i  \prod_{v_k \leq v_i} a'_k) = 0$
for all $i$. This is clear if $a'_i =0$. When $a'_i =1$, there exists $v_j \in S$ with
$v_j \geq v_i$. By the transitivity of the order relation we deduce that $a'_k =1$ for all
$v_k \leq v_i$ and so the equation is satisfied.

 Let $a$ be an element of
  $V(J_P)$. 
 Let $S = f(a)$ and  $a' = g(S)$. We
  want to show that $a = a'$. Suppose that $a'_i =
  1$. Then $\exists\, v_j \in S$ such that $v_i \leq v_j$, and
  therefore $a_i = 1$. By a similar argument, if $a'_i = 0$, then $a_i
  = 0$.
\end{proof}

We now describe the universal reduced Gr\"obner basis of $J_P = J'_P$. 

\begin{proposition}
  \label{prop:gb}
  The universal, reduced Gr\"obner Basis of $J_P$ is the set
  $\mathit{Gb_P}$ of polynomials
  \begin{align*}
    \mathit{gb}_i & = x_i^2 - x_i \qquad \forall\, v_i \in P,\\
    \mathit{gb}_{(j,i)} & = x_i x_j - x_i \qquad \forall\, v_j \leq v_i.
  \end{align*}
\end{proposition}

\begin{proof}
   Lemma~\ref{lem:zeroorone} shows that the elements of $V(J_P)$ are
  strings of $0$'s and $1$'s. The polynomials $x_i^2 - x_i$ are in
  $Gb_P$, and therefore the elements of $V(\mathit{Gb_P})$ are also
  strings of $0$'s and $1$'s. Let $\mathbf{x} = (x_i)_{v_i \in P}$ be
  a string of $0$'s and $1$'s. $\mathbf{x} \in V(J_P)$ if and only if
  $\forall\, v_i \in P,\, (x_i = 0 \Leftrightarrow (\exists\, v_j \leq
  v_i \text{ such that } x_j = 0))$. But this is equivalent to
  $\mathbf{x} \in V(\mathit{Gb_P})$. Then, $\mathit{Gb_P}$ is
  zero-dimensional, and contains a square-free univariate polynomial
  in each variable ($\mathit{gb_i}$). Therefore it is also
  radical. This shows that the ideal generated by $\mathit{Gb_P}$ coincides
  with $J_P$.
  
  We now prove that $\mathit{Gb_P}$ is a Gr\"obner Basis for any
  monomial order $<$. Recall that, given $<$ and a non-zero polynomial
  $p$, $LT_<(p)$ denotes the largest term of $p$, with respect to $<$.
  Clearly, $ LT_<(\mathit{gb}_i) = x_i^2$ and
  $LT_<(\mathit{gb}_{(j,i)}) = x_i x_j$.  Given any two polynomials in
  the set $\mathit{Gb_P}$, we show that their $S$-polynomial is
  divisible by the polynomials in $\mathit(Gb_P)$. If we let
  $p=x_ix_j-x_i$ and $q = x_kx_\ell-x_k$ be two polynomials in
  $\mathit{Gb_P}$, all possible combinations of the indices $i$, $j$,
  $k$ and $\ell$ boil down to the following non-trivial possibilities
  for $(p,q)$ (with $i,j,k,\ell$ all different):

  \begin{enumerate}
  \item $(x_i^2- x_i,x_kx_i- x_k)$ or $(x_ix_j- x_i,x_kx_j- x_k)$
    $\Rightarrow S(p,q) = 0$.
  \item $(x_i^2- x_i,x_ix_\ell- x_i) \Rightarrow S(p,q) =
    \mathit{gb}_i - \mathit{gb}_{(\ell,i)}$.
  \item $(x_ix_j- x_i,x_ix_\ell- x_i) \Rightarrow S(p,q) =
    \mathit{gb}_{(j,i)} - \mathit{gb}_{(\ell,i)}$.
  \item $(x_ix_j- x_i,x_kx_i- x_k) \Rightarrow S(p,q) =
    \mathit{gb}_{(j,k)} - \mathit{gb}_{(i,k)}$.
  \item $(x_i x_j- x_i,x_j x_\ell- x_j) \Rightarrow S(p,q) =
    \mathit{gb}_{(\ell,i)} - \mathit{gb}_{(j,i)}$.
  \item $(x_i^2- x_i,x_k^2- x_k)$ or $(x_i^2- x_i,x_kx_\ell- x_k)$ or
    $(x_ix_j- x_i,x_kx_\ell- x_k)$. In all three cases, since the
    leading monomials of $p$ and $q$ are coprime, $S(p,q)$ is
    divisible by $(p,q)$.
  \item $(x_ix_j- x_i,x_jx_i- x_j)$. This can only hold if $v_i \leq
    v_j$ and $v_j \leq v_i$, that is, $v_i = v_j$.
  \end{enumerate}

  In cases $4$ and $5$ above, we know that $\mathit{gb}_{(j,k)}$ and
  $\mathit{gb}_{(\ell,i)}$, respectively, are in $\mathit{Gb_P}$,
  because a partial order relation is \emph{transitive}. Therefore,
  $\mathit{Gb_P}$ is a Gr\"obner Basis.

  Finally, none of the polynomials are redundant, all of the leading
  coefficients are one, and the ``other'' monomial in each polynomial
  of $\mathit{Gb_P}$ has degree $1$, so it cannot be divisible by any
  leading monomial of $\mathit{Gb_P}$. Therefore, $\mathit{Gb_P}$ is
  a reduced universal Gr\"obner Basis of $J_P$.
\end{proof}

We can count the antichains of $P$ by studying $J_P$. We have seen
that $|{\mathcal A}(P)| = |V(J_P)|$. It is
well-known ({\cite[Theorem~2.2.10]{clo}}) that as $J_P$ is radical, it holds
that
  \begin{equation*}
    |V(J_P)| = \dim_\KK(\KK[\mathbf{x}]/J_P).
  \end{equation*}
The Hilbert Series algorithm could help us to compute
$\dim_\KK(\KK[\mathbf{x}]/J_P)$, but it requires that the ideal $J_P$
be homogeneous, which is not the case. This is circumvented by
considering an initial ideal of $J_P$. If $<$ is a monomial order and
$I$ is an ideal, the initial ideal of $LT_<(I)$ is defined by
\begin{equation*}
  LT_<(I) = \langle LT_<(p), p \in I \rangle.
\end{equation*}
By {\cite[Chapter 5, Section 3]{iva}}
  \begin{equation*}
    \dim_\KK(\KK[\mathbf{x}]/I) =
    \dim_\KK(\KK[\mathbf{x}]/\mathit{LT}_<(I)).
  \end{equation*} 
In particular, we have the following equality
\begin{equation}
  |{\mathcal A}(P)| =   \dim_\KK(\KK[\mathbf{x}]/LT_<(J_P)).
\end{equation}

Let $\mathit{Gb_P}$ be the (universal)
Gr\"obner basis of $J_P$ in the statement of
Proposition~\ref{prop:gb}.
By the definition of a Gr\"obner Basis, it holds that for any
monomial order, 
 $ \mathit{LT}_<(J_P) = \langle \mathit{LT}_<(g), g \in \mathit{Gb_P}
  \rangle$.
Note that this initial ideal
has the same structure of the ideals $I_G$ in Section~\ref{sec:hilbert}.
In fact, it equals $I_{G(P)}$.

\section{Independent sets in bipartite Co\-hen-Ma\-cau\-lay graphs}
\label{sec:cohen}

Let $G$ be a graph, and $I'_G$ its edge ideal. We say that $G$ is a
Cohen-Macaulay graph if $\KK[\mathbf{x}]/I'_G$ is a Cohen-Macaulay
$\KK[x]$-module. The quotient $\KK[\mathbf{x}]/I_G$ is always
Cohen-Macaulay, because $I_G$ is zero-dimensional. Cohen-Macaulay
rings and modules are extensively studied in \cite{brunsherzog}, and
the article \cite{cmgraphs} covers Cohen-Macaulay graphs.

Not every graph is Cohen-Macaulay, of course. For example, the path of
length three (see Figure~\ref{fig:P3}) has the edge ideal $J_{P_3} =
\langle x_1x_2,x_2x_3\rangle$, defined in $\KK[x_1,x_2,x_3]$. The
quotient $\KK[x_1,x_2,x_3]/J_{P_3}$ is not Cohen-Macaulay.  It is not
even equidimensional, since the zero set of $J_{P_3}$ consists of the
plane $x_2 = 0$, together with the line $x_1 = x_3 = 0$.
\begin{figure}[ht]
  \centering
  \includegraphics{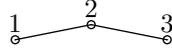}
  \caption{The path of length three $P_3$}
  \label{fig:P3}
\end{figure}

One particularly interesting subfamily of Cohen-Macaulay graphs are
\emph{bipartite Cohen-Macaulay graphs}. 

\begin{definition}
  Let $G = (V_1 \sqcup V_2,E )$ be a bipartite graph. Then $G$ is a
  Cohen-Macaulay graph if and only if $\KK[\mathbf{x}]/I'_G$ is a
  Cohen-Macaulay $\KK[x]$-module.
\end{definition}

There is an equivalent characterization, given by the following
result.

\begin{theorem}[\cite{herhibi}]
  Let $G=(V_1 \sqcup V_2,E)$ be a bipartite graph. We say that $G$ is
  a Cohen-Macaulay graph if $|V_1| = |V_2|$, and the vertices $V_1 =
  \{x_1,\ldots,x_n\}$ and $V_2 = \{y_1,\ldots,y_n\}$ can be labeled in
  such a way that
  \begin{enumerate}
  \item $(x_i,y_i) \in E$ for all $i=1,\ldots,n$;
  \item if $(x_i,y_j) \in E$, then $i \leq j$;
  \item if $(x_i,y_j)$ and $(x_j,y_k)$ are edges, then $(x_i,y_k)$ is
    also an edge.
  \end{enumerate}
\end{theorem}

There are two ways of seeing a bipartite Cohen-Macaulay graph $G=(V_1 \sqcup V_2,E)$ as a
poset. The obvious way is to set the following partial order on the nodes of
$G$: $x \leq y$ if and only if $x = y$ or $x \in V_1$, $y \in V_2$ and $(x,y)$ is
an edge of $G$. That is, one chooses one of the parts as the ``upper''
one.

The other way, which we will consider here, involves a different construction.
Let $G=(V_1 \sqcup V_2,E)$ be a bipartite Cohen-Macaulay graph. We
define a poset $P_G$ as follows. The elements of $P_G$
are those of $V_1$. Given $x_i$ and $x_j$, we set $x_i \leq x_j$ if
and only if the edge $(x_i,y_j)$ is in $E$. From the transitivity of
bipartite Cohen-Macaulay graphs, we see that $P_G$ is a poset.

Conversely, let $P$ be a poset, with elements
$x_1,\ldots,x_r$. We build a bipartite graph $G_P = (V,E)$ as follows. We set $V = V_1
\sqcup V_2$, with $V_1 = \{x_1,\ldots,x_r\}$ and $V_2 =
\{y_1,\ldots,y_r\}$. We put the edges $(x_i,y_i)$ in $E$ for all $i$,
and we have the edge $(x_i,y_j)$ if and only if $x_i \leq x_j$ in
$P$. In this case, the transitivity of $\leq$ ensures that $G_P$ is a
bipartite Cohen-Macaulay graph.

The following lemma is straightforward.

\begin{lemma} \label{lem:inverse}
  The two transformations
  \begin{equation*}
    P \mapsto G_P \quad \text{and} \quad  G \mapsto P_G
  \end{equation*}
   are inverses.
\end{lemma}

We now compare the independence polynomial of a bipartite Cohen-Macaulay
graph $G$ with the antichain polynomial of the poset $P_G$.

\begin{lemma} \label{lem:bip}
  Let $I(G,x)$ be the independence
  polynomial of a bipartite Cohen-Macaulay graph $G$ and let $A(P_G,x)$ 
  be the antichain polynomial of its associated poset  $P_G$. Then
  \begin{equation*}
    I(G,x) = A(P_G,2x).
  \end{equation*}
\end{lemma}

  \begin{proof}
The construction
outlined above expands every element of the poset $P_G$
into a segment in the bipartite Cohen-Macaulay graph $G$.
An antichain $S$ of size $k$ in 
  $P_G$ gives rise to $2^k$ independent sets
  of size $k$ in the bipartite graph $G$, since we can replace
  any $x_i \in S$ by either the node $x_i$ or the node $y_i$ of G.
  It is clear that any independent set of $G$ can be seen in
  this way for a unique antichain $S$ of $P_G$.
\end{proof}

\section{Complexity results} \label{sec:complex}

Is it classically known that it is not possible to count the number of antichains of
a general poset (that is,
to evaluate its antichain polynomial at $1$) in polynomial time unless $P = \#P$ \cite{provanball}.
We extend this result in Theorem~\ref{thm:evalantihard} to the
evaluation at any non-zero rational number $t$, 
by a translation and specialization of \cite[Theorem 2.2]{salvador} to
the context of finite posets. We then use our previous results to deduce in Corollaries~\ref{cor:Hilberthard}
and~\ref{cor:CMhard} the hardness of evaluating the Hilbert function of 
initial ideals of zero-dimensional radical ideals and the independence polynomial of
Cohen-Macaulay bipartite graphs. 

We start with some definitions. 

\begin{definition}  We 
  define the {lexicographic product} poset $P_1[P_2]$ of two finite
  posets $P_1$ and $P_2$ as the set $P_1 \times P_2$,
  ordered by the relation $(x,i) \leq (y,j)$ if $x \leq y \land x = y
  \Rightarrow i \leq j$.
  Similarly, we define the lexicographic product graph $G_1[G_2]$ of
  two graphs as the set $G_1 \times G_2$ with $(i,j)$ adjacent to
  $(k, l)$ iff $i$ is adjacent to $k$ or if $i=k$ and $j $ is adjacent to
  $l$.
\end{definition}

It is easy to check that $P_1[P_2]$ is indeed a poset. 

Given a natural number $m$, denote by $\tilde{K}_m$ the poset given by the set
  $\{1,\ldots,m\}$, ordered with the usual $\leq$ relation.
The associated comparability graph is the complete graph $K_m$ in
$m$ nodes, whose independence polynomial equals $I(K_m,x) = 1 + mx$.

It is straightforward to check that the comparability
graph of the lexicographic product $P_1[P_2]$ of two posets
equals the lexicographic product $G(P_1)[G(P_2)]$ of the respective
comparability graphs. We therefore have:

\begin{lemma} \label{lem:salvador}
For any poset $P$ and $m \in {\mathbb N}$, the comparability graph
of the lexicographic product $P[\tilde{K}_m]$ equals the lexicographic product of the graphs
$G(P)[K_m]$.
\end{lemma}

We are now ready to prove the following theorem:
 
\begin{theorem}
  \label{thm:evalantihard}
  Evaluating the antichain polynomial of any finite poset $P$ at any 
  non-zero rational number $t$ is $\#P$-hard. 
\end{theorem}
\begin{proof}
We mimic the arguments in \cite[Theorem 2.2]{salvador}.
Suppose, on the contrary, that given any poset $P$ on $n$ vertices, there exists an $O(n^k)$-algorithm to compute 
$A(P,t)$ for some constant $k$.
Then, given a poset $P$ with $n$ vertices, consider the posets
$P[\tilde{K}_m]$  for $m=1, \dots, n+1$.
It follows from Lemma~\ref{lem:salvador} that we can use
 the reasoning in \cite[Theorem 2.2]{salvador} to deduce that
 that $A(P,mt) = A(P[\tilde{K}]_m, t)$.  In fact, by
   \cite[Theorem 1]{brown}), $A(P[\tilde{K}]_m, t) = A (P, A(\tilde{K}_m,t) -1) =
   A(P, mt)$. As the posets
  $P[\tilde{K}_m], m = 1, \dots, n+1$ can be constructed in polynomial time from the data 
  of $P$, it would be possible to compute $A(P, mt)$ in polynomial time for $m=1, \dots, n+1$.
  But then, the coefficients $i_j$ of $A(P,x) = \sum_{j=0}^n i_j x^j$ 
  would be computed in polynomial time by solving the $(n+1) \times (n+1)$ linear system with 
  invertible matrix
  $M= (M_{ij})$ given by $M_{ij} = (j t)^{i-1}, \, i, j =1, \dots, n+1$. 
  
  It follows that the number of antichains $|\mathcal A (P)|$ of $P$ would be computable in
  polynomial time by adding $\sum_j i_j$. But this counting problem is $\#P$-complete
   \cite{provanball}.
\end{proof}

Combining this complexity results with the algebraic results of the previous sections,
we have the following two corollaries.

\begin{corollary}
   \label{cor:Hilberthard}
   No algorithm can evaluate the Hilbert Series at a fixed non-zero rational number
   $t$ in polynomial time when applied to initial ideals of radical zero-dimensional
   complete intersections, unless $\#P = P$.
 \end{corollary}

\begin{proof}
By the results of Sections~\ref{sec:hilbert} and~\ref{sec:posets-gb}, the Hilbert Series of
the initial ideal $I_{G(P)}$  of the radical
zero-dimensional ideal $J_P$ associated to any poset, equals the antichain polynomial
$A(P,x)$.  The result follows from Theorem~\ref{thm:evalantihard}.
\end{proof}

\begin{corollary}
  \label{cor:CMhard}
  There can be no polynomial algorithm to evaluate at any non-zero rational
  number $t$ the
  independence polynomial of bipartite Cohen-Macaulay graphs unless
  $\#P = P$.
\end{corollary}

\begin{proof}
 By Lemma~\ref{lem:bip},   the independence polynomial $I(G,p)$ 
  of a bipartite Cohen-Macaulay graph $G$  and the  
  antichain polynomial $A(P_G,x)$  of its associated poset  $P_G$ are
  related by the equality $ I(G,x) = A(P_G,2x)$. So, any polynomial
  algorithm to evaluate $I(G,t/2)$ in polynomial time for any
  Cohen-Macaulay graph $G$, would allow us 
  to evaluate $A(P,t)$ in polynomial time for any poset by
  Lemma~\ref{lem:inverse}.
  The result now follows from Theorem~\ref{thm:evalantihard}.
\end{proof}

\section{Some experimental observations}
\label{sec:exper}

We tested the three Computer Algebra Systems \cocoa, \singular and
Macaulay 2. The examples we used were the posets consisting of the
power set of $\{1,\ldots,n\}$, ordered by inclusion (Boolean
lattice). Of the three systems, only \cocoa managed to count the
antichains for $n = 7$. 
These numbers (called Dedekind numbers) are
known for $n$ up to $13$. However, those computations required many
hours of supercomputer time \cite{dednums}.

The strategy employed by \cocoa for the Hilbert Numerator algorithm
seems to be generally good. We made some observations about it
in~\cite{PortoAlegre}. We have also tested a recent software package,
EdgeIdeals (\cite{EdgeIdeals}). EdgeIdeals allows us to compute the
Hilbert Series of a modified edge ideal of a graph $G$ by computing
the $f$-vector of the simplicial complex associated with the edge
ideal of $G$. The simplicial complex also contains a description of
the standard monomials of the modified edge ideal of $G$. The
computation of both objects (the $f$-vector and the standard
monomials) was faster using EdgeIdeals for the Boolean lattice, 
compared to the native Macaulay 2 implementation of hilbertSeries, for $n$ up to
$6$.  


\bibliographystyle{plain}
\bibliography{bibliography}
\end{document}